\documentclass[11pt]{amsart}
\usepackage[dvips]{graphicx,color}
\usepackage{xcolor}
\usepackage{amsthm,amssymb,amscd,latexsym,epsfig,xypic,hyperref}
\usepackage[mathscr]{euscript}
\DeclareRobustCommand{\SkipTocEntry}[3]{}
\makeatletter
\newcommand\@dotsep{4.5}
\def\@tocline#1#2#3#4#5#6#7{\relax
  \ifnum #1>\c@tocdepth 
  \else
    \par \addpenalty\@secpenalty\addvspace{#2}%
    \begingroup \hyphenpenalty\@M
    \@ifempty{#4}{%
      \@tempdima\csname r@tocindent\number#1\endcsname\relax
    }{%
      \@tempdima#4\relax
    }%
    \parindent\z@ \leftskip#3\relax \advance\leftskip\@tempdima\relax
    \rightskip\@pnumwidth plus1em \parfillskip-\@pnumwidth
    #5\leavevmode #6\relax
    \leaders\hbox{$\m@th
      \mkern \@dotsep mu\hbox{.}\mkern \@dotsep mu$}\hfill
    \hbox to\@pnumwidth{\@tocpagenum{#7}}\par
    \nobreak
    \endgroup
  \fi}
\makeatother

\DeclareFontFamily{OT1}{rsfs}{}
\DeclareFontShape{OT1}{rsfs}{n}{it}{<-> rsfs10}{}
\DeclareMathAlphabet{\curly}{OT1}{rsfs}{n}{it}

\renewcommand\O{\mathcal O}
\newcommand\LL{\mathbb L}
\newcommand\PP{\mathbb P}

\newcommand\C{\mathbb C}
\newcommand\Q{\mathbb Q}
\newcommand\R{\mathbb R}
\newcommand\Z{\mathbb Z}
\renewcommand\t{\mathfrak t}

\newcommand\udot{^{\bullet}}
\renewcommand\;{\hspace{1pt}}
\newcommand{\rt}[1]{\stackrel{#1\,}{\longrightarrow}}
\newcommand{\Rt}[1]{\stackrel{#1\,}{\longrightarrow}}
\newcommand{\RT}[2]{\xymatrix@C=#1pt{\ar[r]^{#2}&}}
\newcommand\To{\longrightarrow}
\newcommand\into{\hookrightarrow}
\newcommand\ito{\ar@{^{ (}->}[r]}
\newcommand{\Into}{\ensuremath{\lhook\joinrel\relbar\joinrel\rightarrow}}

\renewcommand\_{^{}_}
\newcommand\take{\backslash}

\newfont{\bigtimesfont}{cmsy10 scaled \magstep5}
\newcommand{\bigtimes}{\mathop{\lower0.9ex\hbox{\bigtimesfont\symbol2}}}

\newcommand\Gr{\operatorname{Gr}}
\newcommand\rk{\operatorname{rank}}
\newcommand\vd{\operatorname{vd}}

\newcommand\vir{\operatorname{vir}}

\newcommand\Bl{\operatorname{Bl}}

\newcommand\Hom{\operatorname{Hom}}

\newcommand\Ext{\operatorname{Ext}}

\newcommand\Spec{\operatorname{Spec}\;}
\newcommand\Sym{\operatorname{Sym}}

\newcommand\Ob{\operatorname{Ob}}

\newcommand\Crit{\operatorname{Crit}\;}

\newcommand\beq[1]{\begin{equation}\label{#1}}
\newcommand\eeq{\end{equation}}
\newcommand\beqa{\begin{eqnarray*}}
\newcommand\eeqa{\end{eqnarray*}}

\makeatletter \@addtoreset{equation}{section} \makeatother

\newtheorem{thm}[equation]{Theorem}
\newtheorem{lem}[equation]{Lemma}

\newtheorem{prop}[equation]{Proposition}
\newenvironment{rmk}{\noindent\textbf{Remark}.}{}

\theoremstyle{definition}
\newtheorem{eg}{Example}[section]

\DeclareRobustCommand{\SkipTocEntry}[4]{}

\begin{document}

\title{Virtual signed Euler characteristics}
\author{Yunfeng Jiang and Richard P Thomas}

\begin{abstract}
%
%
Roughly speaking, to any space $M$ with perfect obstruction theory we associate a space $N$ with \emph{symmetric} perfect obstruction theory.
It is a cone over $M$ given by the dual of the obstruction sheaf of $M$, and contains $M$ as its zero section. It is locally the critical locus of a function.

More precisely, in the language of derived algebraic geometry, to any quasi-smooth space $M$ we associate its $(-1)$-\mbox{shifted cotangent bundle $N$.} \vskip 5pt

By localising from $N$ to its $\C^*$-fixed locus $M$ this gives five notions of virtual signed Euler characteristic of $M$:
\begin{enumerate}
\item The Ciocan-Fontanine-Kapranov/Fantechi-G\"ottsche signed virtual Euler characteristic of $M$ defined using its own obstruction theory,
\item Graber-Pandharipande's virtual Atiyah-Bott localisation of the virtual cycle of $N$ to $M$,
\item Behrend's Kai-weighted Euler characteristic localisation of the virtual cycle of $N$ to $M$,
\item Kiem-Li's cosection localisation of the virtual cycle of $N$ to $M$,
\item $(-1)^{\vd}$ times by the topological Euler characteristic of $M$.
\end{enumerate}
Our main result is that (1)=(2) and (3)=(4)=(5). The first two are deformation invariant while the last three are not.
\end{abstract}

\maketitle \vspace{-5mm}
\tableofcontents

\section{Introduction}
Spaces (schemes or Deligne-Mumford stacks) $M$ with perfect obstruction theory $E\udot\to\LL_M$ \cite{BF} have a virtual dimension
$$
\vd=\rk(E\udot)
$$
and a natural $\vd$-dimensional virtual cycle
$$
[M]^{\vir}\in A_{\vd}(M)\To H_{2\!\;\vd}(M)
$$
over which one can integrate cohomology classes to give invariants. Such spaces usually arise from moduli problems.

Spaces $N$ with \emph{symmetric} obstruction theory \cite{BehDT} carry a \emph{0-dimensional} virtual cycle over which one can integrate 1 to give a virtual signed Euler characteristic of $N$. When $N$ is smooth and proper, the invariant is
$$
\int_{[N]^{\vir}}1\ =\ c\_{\dim N}(\Omega_N)\,=\,(-1)^{\dim N}e(N),
$$
where $e(\,\cdot\,)$ denotes the topological Euler characteristic. Even when $N$ is singular (but still compact) its invariant is a \emph{weighted} Euler characteristic
$$
\int_{[N]^{\vir}}1\ =\ e\big(N,\chi^N\big),
$$
weighted by the Behrend function $\chi^N$ \cite{BehDT} -- a canonical constructible function $N\to\Z$.

In this paper we pass from the first setting to the second to define various notions of signed virtual Euler characteristic of $M$. From $M$ we construct a space $N$ with \emph{symmetric} perfect obstruction theory. It is a cone over $M$, containing $M$ as its zero section.

When $M$ is smooth of the correct virtual dimension, $N=M$. When $M$ is merely smooth (or a local complete intesection), $N$ is the total space of the dual of the obstruction bundle $\Ob_M=h^1((E\udot)^\vee)$ over $M$,
$$
N=\Ob_M^*\To M.
$$
For general $M$ it is the cone $\Spec\Sym\udot\Ob_M\to M$; see Section \ref{constr} for details.

$N$ actually satisfies a stronger property than having a symmetric obstruction theory: it is locally the critical locus of a function. The model is the following. Locally we can write $M$ as the zero locus of a section $s$ of a vector bundle $E\rt\pi A$ over a smooth ambient space $A$, such that the perfect obstruction theory of $M$ is the natural one
$$
E\udot\ =\ \big\{E^*|_M\Rt{ds}\Omega_A|_M\big\}\To\LL_M.
$$
Consider $s$ as a function $\widetilde s$ on the total space Tot$\;(E^*)$ of the bundle $E^*$, linear on the fibres. Then $N$ is the critical locus of $\widetilde s$,
\beq{critN}
N=\,\Crit(\widetilde s\;).
\eeq
In the language of derived algebraic geometry, $N$ is $(-1)$-shifted symplectic \cite{PTVV} because it is the $(-1)$-shifted cotangent bundle of the quasi-smooth space underlying $(M,E\udot)$. \medskip

\begin{rmk} In fact there is an obstruction to globalising the symmetric obstruction theory on $N$ given by the local construction \eqref{critN}. This vanishes if we make the assumption that $(M,E\udot)$ admits the structure of a quasi-smooth derived space; we can then set $N:=T^*_M[-1]$. This is also quasi-smooth and so gives rise to an underlying space with perfect (symmetric) obstruction theory.

In an earlier draft of this paper we applied \cite{Sch} to $E\udot$ to construct such a quasi-smooth derived structure on $M$. Since it now seems there are some problems with \cite{Sch}, we instead make the \emph{assumption} that $(M,E\udot)$ comes from a quasi-smooth derived space. In the examples which occur in nature $M$ is a moduli space with a canonical quasi-smooth derived structure.
\end{rmk} \bigskip

Although $N$ is noncompact it carries a natural $\C^*$-action -- scaling the fibers of $N\to M$ -- with compact fixed locus $M$. Therefore we can define a numerical invariant from its 0-dimensional virtual class by localising to $M$. There are many ways to do this localisation, but whenever $M$ is smooth of the correct dimension they all give the same answer
$$
c\_{\;\dim M}(\Omega_M)\ =\ (-1)^{\dim M}e(M).
$$
However, in general there is no obvious reason why they should all agree. So a priori, we get the following four different notions of virtual signed Euler characteristic of $M$.\footnote{See Section $3.n$ for the definition of the $n$th virtual signed Euler characteristic.}
\begin{enumerate}
\item The Ciocan-Fontanine-Kapranov/Fantechi-G\"ottsche signed virtual Euler characteristic $\int_{[M]^{\vir}}c_{\vd}(E\udot)$ of $M$ \cite{CK,FG}. This is defined entirely in terms of $M$ and its obstruction theory, without use of $N$.
\item Graber-Pandharipande's virtual Atiyah-Bott localisation \cite{GP} of the virtual cycle of $N$ to $M$.
\item Kai localisation of the weighted Euler characteristic $e\big(N,\chi^N\big)$ to $M$. Here $\chi^N$ denotes the Behrend function \cite{BehDT} of $N$. Since $N\setminus M$ carries a \emph{free} $\C^*$-action which preserves $\chi^N$, its contribution vanishes, giving the localisation $e\big(N,\chi^N\big)=e\big(M,\chi^N|_M\big)$.
\item Kiem-Li's cosection localisation \cite{KL} of the virtual cycle of $N$ to $M$. The $\C^*$-action defines a canonical Euler vector field on $N$. By the symmetry of $N$'s obstruction theory, this defines a cosection $\Ob_N\to\O_N$. In such situations Kiem and Li give a way to localise the virtual cycle of $N$ to the zeros of the cosection, i.e. to $M$. 
\end{enumerate}

If $N$ were compact, these definitions would all give the same answer $\int_{[N]^{\vir}}1$. However, even when $M$ is smooth of too high a dimension, it is easy to calculate that they can give different answers. Suppose $\Ob_M$ admits a regular section cutting out a smooth $\vd$-dimensional representative $M^{\vir}\subset M$ of the virtual cycle, then (1) and (2) equal
$$
(-1)^{\vd\,}e(M^{\vir}),
$$
whereas (3) and (4) give
$$
(-1)^{\vd\,}e(M).
$$
This turns out to be part of a more general phenomenon.

\begin{thm} Suppose that $M$ is a projective scheme with perfect obstruction theory $E\udot$ arising as $\pi_0$ of a quasi-smooth derived scheme. Then $(1)=(2)$ and $(3)=(4)=(-1)^{\vd\;}e(M)$. The first two are deformation invariant while the last two are not.
\end{thm}

\noindent\textbf{Plan.} In Section \ref{constr} we construct $N$ and its obstruction theory. After giving a local description we are forced (briefly!) to use derived algebraic geometry to give a definitive global construction in \eqref{defnNN}. We also give an example showing how the construction arises in nature from moduli of sheaves. The expert can skip straight to Section \ref{four}, where we review the definitions of the four virtual signed Euler characteristics. We also prove the easiest relations between them: $(1)=(2)$ and, through examples, $(1)\ne(3)$ in general. Section \ref{equal} is devoted to proving $(3)=(4)$, while Section \ref{ordinary} shows that $(3)$ gives the ordinary signed Euler characteristic of $M$. \medskip

\noindent\textbf{Acknowledgements.} It was Davesh Maulik who explained to us that (3) should equal $(-1)^{\vd\;}e(M)$, for which we are very grateful.

This paper was motivated by the papers \cite{CL} and \cite[Section 2.2]{VW} and by ``cotangent field theories" in the language of Costello \cite{Co}. We would like to thank Paolo Aluffi, Kai Behrend, Jun Li, J\o rgen Rennemo and Ed Segal for useful conversations, Jon Pridham, Timo Sch\"urg, Bertrand To\"en and Gabriele Vezzosi for discussions about \cite{Sch}, and two thorough referees for many suggested improvements.

The first author was partially supported by NFGRF, University of Kansas, and a Simons Foundation Collaboration Grant 311837, and the second author by an EPSRC Programme Grant EP/G06170X/1.
\medskip

\noindent\textbf{Notation.} Throughout we work for simplicity with a complex projective scheme $M$ with perfect obstruction theory $E\udot\to\LL_M$. In fact all the arguments extend with obvious minor changes to the case when $M$ is a Deligne-Mumford stack, on replacing our use of the Kashiwara index theorem \cite{Ka} with Maulik-Treumann's orbifold version \cite{MT} in Section \ref{equal}.

From \eqref{derived} we have to assume that $(M,E\udot)$ is the truncation of a quasi-smooth derived scheme $M^{\mathrm{der}}$. We use the notation $E^{-i}:=E_i^*$ for dual vector bundles, reserving ${}^\vee$ for the \emph{derived} dual of coherent sheaves and complexes.

\section{The construction of $N$} \label{constr}
\subsection*{Abelian cones}
For $F$ a coherent sheaf over $M$, there is an associated cone
$$
C(F):=\Spec\Sym^\bullet F\Rt{\pi_F}M
$$
over $M$. Cones of this form are called \emph{abelian} in \cite[Section 1]{BF}. The grading on $\Sym^\bullet F$ endows $C(F)$ with a $\C^*$-action 
$$\ \C^*\times C(F)\To C(F)$$
induced by the map
$$\Sym^\bullet F[x,x^{-1}]\longleftarrow\Sym^\bullet F\ \ $$
that takes $s\in \Sym^i F$ to $sx^i$. Its fixed locus is the zero section $M\subset C(F)$ defined by the ideal $\Sym^{\ge1}F$.

When $F$ is locally free $C(F)=\mathrm{Tot}(F^*)$ is the total space of the dual vector bundle. More generally, for any $F$, the fibre of $C(F)$ over a closed point $p\in M$ is the vector space $(F|_p)^*$. In fact $C(F)$ represents the functor from $M$-schemes to sets that takes $f\colon S\to M$ to $\Hom_S(f^*F,\O_S)$.

\begin{lem} \label{cotan}
Given a locally free resolution $E_0\rt\phi E_1\to F\to0$, the cone $C(F)\to M$ inherits a perfect \emph{relative} obstruction theory over $M$ given by
$$
\big\{\!\!\xymatrix@=25pt{\pi_F^*E_0 \ar[r]^{\pi_F^*\phi\ }& \pi_F^*E_1}\!\!\big\}\To\LL_{C(F)/M}.
$$
Here $\pi_F^*E_1$ is in degree 0. In particular, taking $h^0$ gives $\Omega_{C(F)/M}\cong\pi_F^*F$.
\end{lem}

\begin{proof}
The resolution gives an exact sequence
$$
\phi(E_0)\otimes\Sym^{\bullet\;-1}E_1\To\Sym^\bullet E_1\To\Sym^\bullet F\To0.
$$
That is,
$$
\Spec\Sym^\bullet F\ \subset\ \Spec\Sym^\bullet E_1
$$
with ideal generated by $\phi(E_0)$. Letting $\tau$ denote the tautological section of $\pi_{E_1}^*E^{-1}$, this says that\smallskip
\beq{sentence}
C(F) \text{ is cut out of }C(E_1)=\mathrm{Tot}\;(E^{-1}) \text{ by the section $\pi_{E_1}^*\phi^*(\tau)$ of } \pi_{E_1}^*E^0.
\eeq
Therefore by a standard construction which we review below, $C(F)$ inherits a natural perfect \emph{relative} obstruction theory over $M$, given by
$$
\big\{\!\!\xymatrix@=25pt{\pi_F^*E_0 \ar[r]^{\pi_F^*\phi\ } & \pi_F^*E_1}\!\!\big\}\To\LL_{C(F)/M}.
$$
More generally, suppose we have a \emph{smooth map}\footnote{We do not require $A$ or $B$ to be smooth; just the map $A\to B$ which in our application is $C(E_1)=\mathrm{Tot}(E^{-1})\to M$. Over this $(E,s):=(\pi_{E_1}^*E^0,\pi_{E_1}^*\phi^*(\tau))$ cuts out $X=C(F)$.} $A\to B$ and a section $s$ of a vector bundle $E\to A$. Then the zero scheme $X$ of $s$ inherits a natural perfect \emph{relative} obstruction theory
\beq{relpot}
\big\{E^*|_X\Rt{ds}\Omega_{A/B}|\_X\big\}\To\LL_{X/B}\;.
\eeq
To describe the maps, we use the embedding of $A$ as the zero section of Tot$\;(E)$, and then identify $E^*$ with the first factor of
$$
\Omega_{\;\mathrm{Tot}\;(\hspace{-1pt}E)/B}\big|\_A\ \cong\ E^*\oplus\Omega_{A/B}\;.
$$
This maps to the relative cotangent sheaf of the graph $\Gamma_s\subset\mathrm{Tot}\;(E)$ restricted to $A\cap\Gamma_s=X$, giving a map
\beq{2nd}
E^*\big|\_X\To\Omega_{\Gamma_s/B}\big|\_X\;.
\eeq
Via the isomorphism $\Gamma_s\cong A$ (by projection) this gives the first arrow of \eqref{relpot}. To describe the second we rewrite $\Omega_{\Gamma_s/B}|\_X$ as $\LL_{\Gamma_s/B}|\_X$, which then maps to $\LL_{X/B}$. Its composition with the arrow \eqref{2nd} is zero, and the result is indeed a perfect obstruction theory \cite[Section 6]{BF}.
\end{proof}

\subsection*{Constructing $N$}
We fix a perfect obstruction theory
$$
E\udot\To\LL_M
$$
of virtual dimension
$$
\vd:=\rk(E\udot)
$$
on the complex projective scheme $M$.

Applying the results of the last section to the obstruction sheaf $$\Ob_M\!:=h^1\big((E\udot)^\vee\big),$$ we define $\pi=\pi_N\colon N\to M$ to be the associated abelian cone,\footnote{Another way to describe $N$ is as the coarse moduli space of the vector bundle stack $h^1/h^0\big((E\udot)^\vee\big)$ of \cite[Section 2]{BF}.}
\beq{defnN}
N:=C(\Ob_M)=\Spec\Sym^\bullet(\Ob_M)\Rt\pi M.
\eeq
Writing $E\udot$ as $E^{-1}\to E^0$, we get the exact sequence
$$
E_0\To E_1\To\Ob_M\To0.
$$
Therefore by Lemma \ref{cotan}, $N$ inherits a natural perfect relative obstruction theory over $M$,
\beq{cotanN}
\pi^*(E\udot)^\vee[1]\To\LL_{N/M}\;.
\eeq

\subsection*{Local model} Locally we may choose a presentation of $(M,E\udot)$ as the zero locus of a section $s$ of a vector bundle $E\to A$ over a smooth ambient space $A$, such that the resulting complex
$$
\big\{T_A|_M\Rt{ds}E|_M\big\} \quad\text{is}\quad
\big\{E_0\To E_1\big\}=(E\udot)^\vee.
$$
Therefore, by \eqref{sentence}, $N=C(\Ob_M)$ is cut out of Tot$\;(E^*)|_M$ by the section $\pi_E^*(ds)^*(\tau)$ of $\pi_E^*\Omega_A|_M$. In turn Tot$\;(E^*)|_M$ is cut out of Tot$\;(E^*)$ by $\pi_E^*s$. Therefore the ideal of $N$ in the smooth ambient space Tot$\;(E^*)$ is
\beq{ideal}
\big(\pi_E^*s,\,\pi_E^*(Ds)^*(\tau)\big),
\eeq
where we have chosen any holomorphic connection $D$ on $E\to A$ by shrinking $A$ if necessary.

Thinking of the section $s$ of $E\to A$ as a linear function $\widetilde s$ on the fibres of Tot$\;(E^*)$, we find that its critical locus is $N$.

\begin{prop}\label{locCS}
$N\subset\mathrm{Tot}\;(E^*)$ is the critical locus of the function
$$
\widetilde s\colon\mathrm{Tot}\;(E^*)\to\C\;.
$$
\end{prop}

\begin{proof} The only difficulty       here is notational. One approach is to write
$$\widetilde s=\big\langle\pi_E^*s,\tau\big\rangle=\pi_E^*s^*(\tau)$$
in terms of the tautological section $\tau$ of $\pi_E^*E^*$. Using the connection $D$ on $E$, we differentiate
$$
d\;\widetilde s=(D\tau)(\pi_E^*s)+\tau(\pi_E^*Ds).
$$
This vanishes precisely where both summands vanish, since the first is a vertical one-form and the second is horizontal. Thus $\Crit(\widetilde s)$ has ideal \eqref{ideal}, as required.

Alternatively, one can work in local coordinates $x_i$ for $A$. Trivialising $E$ with a basis of sections $e_j$, we get a dual basis $f_j$ for $E^*$ and coordinates $y_j$ on the fibres of Tot$\;(E^*)$.

Then we can write $s=\sum_js_je_j,\ \tau=\sum_jy_jf_j$ and
$$\widetilde s=\sum_j s_jy_j.$$
Therefore
$$
d\;\widetilde s=\sum_jy_jds_j+\sum_js_jdy_j
=\big\langle\tau,\pi_E^*Ds\big\rangle+\sum_js_jdy_j
$$
with zero scheme defined by the ideal
$$
\big(\pi_{E^*}^*(Ds)^*(\tau),\,\pi_E^*s_1,\,\pi_E^*s_2,\,\ldots\big).
$$
This is the same as \eqref{ideal}.
\end{proof}

\subsection*{Global model}
In particular, $N$ has a natural local \emph{symmetric} perfect obstruction theory \cite{BehDT}.
Globally, we would like to fit the relative obstruction theory $\pi^*(E\udot)^\vee[1]\to\LL_{N/M}$ \eqref{cotanN} of $\pi: N\to M$ together with the obstruction theory $\pi^*E\udot\to\pi^*\LL_M$ for $N$ to give an absolute symmetric perfect obstruction theory for $N$. That is we would like to find $F\udot$ to fill in the diagram
\beq{dotted}
\xymatrix@=16pt{
\pi^*E\udot \ar[d]\ar@{..>}[r] & F\udot\! \ar@{..>}[d]\ar@{..>}[r] & \pi^*(E\udot)^\vee[1] \ar[d] \\
\pi^*\LL_M \ar[r]& \LL_N \ar[r]& \LL_{N/M}.\!}
\eeq
So we need to specify a (symmetric) map $\pi^*(E\udot)^\vee\to\pi^*E\udot$; taking the cone would give $F\udot$ and induce the other arrows.
Letting $K$ denote the cone of the left hand vertical arrow gives the diagram
\beq{910}
\xymatrix@=16pt{
\pi^*(E\udot)^\vee \ar[d]\ar@{..>}[r]& \pi^*E\udot \ar[d] \\
\LL_{N/M}[-1] \ar[r]& \pi^*\LL_M \ar[d] \\
& K,}
\eeq
so the obstruction to filling in the dotted arrow is the vanishing of the down-right-down composition
\beq{compos}
\pi^*(E\udot)^\vee\To K.
\eeq

Denote the down-right composition $\pi^*(E\udot)^\vee\to\LL_{N/M}[-1]\to\pi^*\LL_M$ by
$$
A\in\Hom_N\big(\pi^*(E\udot)^\vee,\pi^*\LL_M\big)\ =\ \Hom_M\big((E\udot)^\vee,\LL_M\otimes\pi_*\O_N\big).
$$
Under the $\C^*$-action on $\pi_*\O_N=\Sym\udot\Ob$, its degree 1 part is
$$
A_1\in\Hom_M\big((E\udot)^\vee,\LL_M\otimes\Ob\big).
$$
By \cite[Section IV.2.3]{Ill}, it is the image of the Atiyah class of $(E\udot)^\vee$ under the map $(E\udot)^\vee[1]\to\Ob$:
$$
\mathrm{At}_{(E\udot)^\vee}\in\Ext^1_M\big((E\udot)^\vee,(E\udot)^\vee\otimes\LL_M\big)\To
\Hom_M\big((E\udot)^\vee,\Ob\otimes\LL_M\big).
$$
We let $\eta$ denote the remaining map $\pi^*\LL_M\to K$ in \eqref{910}; it is the Kodaira-Spencer class of the inclusion of $M$ into the derived thickening $M^\eta$ constructed in \cite[Section 2.1]{Sch}.\footnote{$M^\eta$ is Sch\"urg's first approximation to a quasi-smooth derived scheme giving rise to $M$ and its obstruction theory.}

Therefore the degree 1 part of the obstruction \eqref{compos} is a projection of the product $\eta\circ\mathrm{At}_{(E\udot)^\vee}$ of Atiyah and Kodaira-Spencer classes. This product is the obstruction to lifting the complex $(E\udot)^\vee$ from $M$ to the thickening $M^\eta$, which is identified in \cite{Sch} as the first in a sequence of obstructions to finding a quasi-smooth derived scheme whose truncation is $(M,E\udot)$.

Unfortunately there seems to be no general reason to expect it to vanish unless $(M,E\udot)$ indeed arises from a quasi-smooth derived structure. So from now on we make this assumption.
\medskip

\begin{rmk} In fact all of our virtual signed Euler characteristics are defined by localisation from $N$ to $M$, and so for their \emph{existence} we only require the existence of the perfect obstruction theory $F\udot\to\LL_N$ on the zero section $M\subset N$. Since the zero section splits the projection $\pi$, the horizontal triangles of \eqref{dotted} split there, and
$$F\udot|_M\cong E\udot\oplus(E\udot)^\vee[1]$$
exists just fine. Only our proof that (3)=(4) in Section \ref{equal} currently uses the global obstruction theory of $N$ (and we conjecture the result holds without it). Therefore the reader uncomfortable with derived algebraic geometry can ignore the next section and proceed straight to the definitions of the various Euler characteristics.
\end{rmk}

\subsection*{Construction using derived geometry}
So we use the language of derived algebraic geometry \cite{HAG} for a paragraph.

Hereon in we assume that the projective scheme with perfect obstruction theory $(M,E\udot)$ arises from a quasi-smooth derived scheme $M^{\mathrm{der}}$.
That is,
\begin{itemize}
\item The truncation $\pi_0(M^{\mathrm{der}})$ is $M$,
\item The cotangent complex $\LL_{M^{\mathrm{der}}}$ is perfect of amplitude contained in $[-1,0]$, and
\item Its restriction $\LL_{M^{\mathrm{der}}}\big|_M$ to $M$ is $E\udot$ (with its canonical map to $\LL_M$).
\end{itemize}
Then we define
\beq{derived}
N^{\mathrm{der}}:=T^*_{M^{\mathrm{der}}}[-1]
\eeq
to be the $(-1)$-shifted cotangent bundle of $M^{\mathrm{der}}$. This is also quasi-smooth and (-1)-shifted symplectic \cite{PTVV}, so its truncation
\beq{defnNN}
(N,F\udot):=\big(\pi_0(N^{\mathrm{der}}),\LL_{N^{\mathrm{der}}}\big|_N\big)
\eeq
 is a scheme with symmetric perfect obstruction theory. It coincides with \eqref{defnN} as a scheme.
Letting $\pi^{\mathrm{der}}\colon N^{\mathrm{der}}\to M^{\mathrm{der}}$
denote the projection, the distinguished triangle \cite{HAG}
$$(\pi^{\mathrm{der}})^{*\;}\LL_{M^{\mathrm{der}}}\To \LL_{N^{\mathrm{der}}}\To \LL_{N^{\mathrm{der}}/M^{\mathrm{der}}}$$
restricts to $N\subset N^{\mathrm{der}}$ to give the diagram \eqref{dotted} with the dotted arrows filled in. On the zero section, $F\udot|_M\cong E\udot\oplus(E\udot)^\vee[1]$.

So we now forget all about derived algebraic geometry again, and use only $(N,F\udot)$.

\subsection*{Example: moduli of coherent sheaves on local surfaces}
This construction of $N$ from $M$ arises in nature when $M$ is a moduli space of stable sheaves on a projective surface $S$.

By pushing sheaves forward by the inclusion
$$
S\Into X:=\mathrm{Tot}\;(K_S)
$$
into the canonical bundle of $S$ we get an inclusion of $M$ into the moduli space of stable sheaves on the Calabi-Yau 3-fold $X$. By the usual spectral cover construction, stable sheaves on $X$ (finite over $S$) are the same as stable Higgs pairs
$$
(E,\phi) \text{ on $S$, where }\phi\in\Hom(E,E\otimes K_S).
$$
The cone $N$ is then the open set of Higgs pairs such that the underlying sheaf $E$ is stable on $S$.

Moduli spaces of stable sheaves carry canonical obstruction theories (and in fact derived structures \cite{TVa}). Fixing determinants (for simplicity) and removing the trace part of the obstruction complex gives a perfect obstruction theory on the coarse moduli space. Applying this on $S$ and $X$ gives obstruction theories on $M$ and $N$ compatible with the construction of this section.

Therefore we expect our paper to have applications to the S-duality conjecture of \cite{VW} for a 4-manifold which is a complex surface $S$. However it is still not clear to us which of the two distinct virtual Euler characteristics in this paper one should use on the moduli space of stable sheaves on $S$.

\section{The four virtual signed Euler characteristics}\label{four}

We now describe our four signed Euler characteristics $e_i(M),\ i=1,\ldots,4$. The first uses only $M$ in its construction; the rest all utilise $N$ by localising its virtual cycle to its zero section $M$ in different ways.

\subsection{The signed virtual Euler characteristic of Ciocan-Fontanine-Kapranov and Fantechi-G\"ottsche}\label{CFK-FG}
Thinking of $(E\udot)^\vee$ as the virtual (or derived) tangent bundle of $M$, one can form a virtual Euler characteristic by integrating against the virtual cycle \cite{CK, FG}:
$$
e_{\vir}(M):=\int_{[M]^{\vir}}c_{\vd}\big((E\udot)^\vee\big).
$$
Of more interest to us is the virtual version of the signed Euler characteristic which would be given classically by the top Chern class of the \emph{cotangent} bundle. So we consider $E\udot$ to be the virtual cotangent bundle of $M$ and define
\beq{e1def}
e_1(M):=\int_{[M]^{\vir}}c_{\vd}\big(E\udot\big).
\eeq
This is deformation invariant, and depends only on $M$ and its obstruction theory $E\udot$.

\subsection{Graber-Pandharipande localisation}\label{GP}
We use the virtual Atiyah-Bott localisation of Graber-Pandharipande on the virtual cycle of $N$. Its natural $\C^*$-action has fixed locus $M$, along which the horizontal triangles \eqref{dotted} split,
$$
F\udot\big|_M\ \cong\ E\udot\ \oplus\ (E\udot)^\vee\!\otimes\t^{-1}[1].
$$
Here we have made the $\C^*$-action explicit in the notation; $\t$ denotes the standard weight $1$ representation of $\C^*$.

The first summand is $\C^*$-fixed, and gives the obstruction theory for the fixed locus $M$. The second has $\C^*$-weight $-1$; its dual is called the virtual normal bundle,
$$
N^{\vir}\ \cong\ E\udot\otimes\t\,[-1].
$$
The recipe of \cite{GP} for the localisation of the $\vd=0$ virtual cycle of $N$ is
\beq{GP}
e_2(M):=\int_{[M]^{\vir}}\frac1{e(N^{\vir})}\,.
\eeq
Here $e$ denotes the $\C^*$-equivariant Euler class of $N^{\vir}$. Writing
$$
N^{\vir}=\big\{E^{-1}\otimes\t\to E^0\otimes\t\big\}
$$
with $E^{-1}\otimes\t$ in degree 0, it is defined to be
$$
e(N^{\vir})\ =\ \frac{c_{\mathrm{top}}^{T}(E^{-1}\otimes\t)}{c_{\mathrm{top}}^{T}(E^0\otimes\t)}
$$
in the localised cohomology group $$H_{T}^{*}(M,\Q)\otimes_{\Q[t]}\Q[\![t,t^{-1}]\!]\ \cong\ H^*(M,\Q)\otimes_{\Q}\Q[\![t,t^{-1}]\!].$$
Here we let $t:=c_1(\t)\in H^*(B\C^*,\Q)\cong\Q[t]$ denote the first Chern class of $\t$, the generator of the equivariant cohomology of $B\C^*$.

\begin{prop} The Graber-Pandharipande localised signed Euler characteristic \eqref{GP} equals the signed virtual Euler characteristic \eqref{e1def} of Ciocan-Fontanine-Kapranov/Fantechi-G\"ottsche,
$$e_1(M)=e_2(M).$$
\end{prop}

\begin{proof}
Let $r$ and $s=r+\vd$ denote the ranks of $E^{-1}$ and $E^0$ respectively. Then \eqref{GP} equals
$$
e_2(M)=\int_{[M]^{\vir}}\frac{c_s(E^0)+tc_{s-1}(E^0)+\ldots}{c_r(E^{-1})+tc_{r-1}(E^{-1})+\ldots}\,.
$$
The integrand is homogeneous of degree $s-r=\vd$, so only the $t^0$ coefficient has the correct degree $\vd$ over $M$ to have nonzero integral against $[M]^{\vir}$. Therefore we may set $t=1$ in the above to give
$$
e_2(M)=\int_{[M]^{\vir}}\left[\frac{c(E^0)}{c(E^{-1})}\right]_{\vd}\,,
$$
where $c(\,\cdot\,)$ denotes the total Chern class. But this is \eqref{e1def}.
\end{proof}

\subsection{Kai localisation}\label{Kai}
Behrend \cite{BehDT} defines a constructible function
$$
\chi^N\colon N\To\Z
$$
on any scheme $N$. When $N$ is compact with symmetric obstruction theory, he proves the degree of the virtual cycle is the $\chi^N$-weighted Euler characteristic of $N$:
\beq{weigh}
\int_{[N]^{\vir}}1\,=\,e\big(N,\chi^N\big):=\sum_{i\in\Z}i\,e\big((\chi^N)^{-1}\{i\}\big).
\eeq
Since our $N$ is noncompact, the left hand side of \eqref{weigh} is not defined, but we can use the right hand side instead. Over $N\setminus M$ we get zero since there is a free $\C^*$-action preserving $\chi^N$. Therefore \eqref{weigh} localises to $M$ and we define
$$
e_3(M):=e\big(M,\chi^N|_M\big)=e\big(N,\chi^N\big).
$$

\begin{eg}
It is easy to see that $e_3(M)$ is not deformation invariant -- and not always equal to $e_1(M)$ -- by noting that it can be nonzero even if $M$ has negative virtual dimension. For instance if $M$ is a reduced point carrying a rank $-\vd>0$ obstruction space, then $N=\C^{|\vd|}$ and $e_3(M)=(-1)^{\vd}$. Of course $e_1(M)=0=e_2(M)$ in this situation.
\end{eg}

\begin{eg} \label{egs}
The simplest example in nonnegative virtual dimension is as follows. Let $M$ and its perfect obstruction theory be defined by the section $x^2-a^2$ of the trivial line bundle over $\C=\Spec\C[x]$.

For $a\ne0$, $M=N$ is two reduced points and both $e_1(M)$ and $e_3(M)$ equal 2. By deformation invariance the same is true of $e_1(M)$ when $a=0$. But now $N=\Spec\C[x,y]/(x^2,xy)=\;\Crit(x^2y)$ in $\C^2$ and a computation\footnote{For instance, one can show that the normal cone of $N$ is a copy of $\C$ over the $y$-axis plus a $\C^2$ with multiplicity 2 over the origin. Therefore its signed support \cite[Section 1.1]{BehDT} is twice the origin minus the $y$-axis. These are both smooth so their Euler obstructions are their characteristic functions; adding gives $\chi^N$.} shows that $\chi^N$ equals $-1$ away from the origin and $+1$ at the origin. In particular $e_3(M)=1$ is not deformation invariant.

This example also demonstrates that $e\big(M,\chi^N|_M\big)\ne e(M,\chi^M)$ in general; the latter gives 2 here.

It is a good exercise in the definitions of the Kiem-Li localised invariant $e_4(M)$ of the next section to calculate in these examples and verify it equals $e_3(M)$.
\end{eg}

\subsection{Kiem-Li's cosection localisation}

Fix a 2-term locally free resolution $F^{-1}\to F^0$ of the obstruction theory of $N$. Behrend and Fantechi \cite[Section 5]{BF} define a normal cone
$$
C_N\subset F_1=(F^{-1})^*
$$ 
by pulling back the intrinsic normal cone from $h^1/h^0\big((F\udot)^\vee\big)$. The intersection of $C_N$ with the zero section $0_{F_1}$ gives $[N]^{\vir}$.

The $\C^*$-action on $N$ induces an Euler vector field $v$ on $N$. By the symmetry of the obstruction theory on $N$, we get what Kiem-Li call a \emph{cosection}, i.e. a map
\beq{cose}
\sigma: \Ob_N\,\cong\,\Omega_N\Rt{}\O_N.
\eeq
It is surjective on $N\take M$, so we expect the virtual cycle to vanish away from the zero locus $M$ of $v$. And indeed Kiem and Li define a localised virtual cycle in $A_0(M)$ as follows. \medskip

From the cosection map \eqref{cose} we get the composition
\beq{cosec}
F_1\To\Ob_N\Rt{v}\O_N
\eeq
whose image is the ideal sheaf of $M\subset N$. Therefore its pull back to the blow up of $N$ along $M$,
$$
p\colon\!\Bl_M(N)\To N,
$$
has image $\O_{\Bl_M(N)}(-E)\subset\O_{\Bl_M(N)}$ (where $E$ is the exceptional divisor). Letting $K$ denote its kernel, we get the exact sequence of vector bundles
\beq{kern}
0\To K\To p^*F_1\To\O_{\Bl_M(N)}(-E)\To0.
\eeq
We can write
$$
C_N\,=\,C_1+C_2\,=\,C_1+p_*(\;\overline C_2),
$$
where $C_1$ is supported on $M$, by setting $\overline C_2$ to be the proper transform of $C_N$. Then Kiem and Li show that $\overline C_2$ lies (set- or cycle-theoretically) in $K$, so they define
\beq{zeroc}
[N]^{\vir}_{\mathrm{loc}}:=0_{F_1}^!(C_1)-p_*\Big([E].(0_K^!\overline C_2)\Big)\,\in\,A_0(M).
\eeq
The first term is already supported on $M$. The second is too because we intersect with $E$. This gives our fourth virtual signed Euler characteristic,
\beq{e4}
e_4(M):=\int_{[N]^{\vir}_{\mathrm{loc}}}1\;,
\eeq
the length of the zero cycle \eqref{zeroc} in $A_0(M)$. \medskip
%

\begin{rmk}
Example \ref{egs}, together with the equality $e_3(M)=e_4(M)$ to be proved in the next Section, show that $e_4(M)$ need \emph{not} be deformation invariant as we deform $(M,E\udot)$. This may seem surprising since Kiem and Li prove a deformation invariance result \cite[Theorem 5.2]{KL}, but under an assumption which does not hold here.

Given a family $(M_t,E_t\udot)$ over a smooth base $B\ni t$, we get an associated family $(N_t,F_t\udot)$. There is an obstruction map from the (pullback to $N$ of) $T_tB$ to the obstruction sheaf $\Ob_{N_t}$. The condition imposed in \cite{KL} amounts to asking that the cosection be zero on such obstructions, and this is what fails in 
Example \ref{egs}.

As the referee pointed out, another way to look at this is that the cosection \ref{cose} is not $\C^*$-invariant (it has weight 1). If it were invariant then $e_4(M)$ would be deformation invariant by the recent paper \cite{CKL}.
\end{rmk}

\section{Equality of Kai and Kiem-Li localisations}\label{equal}

Since $M$ is projective we may fix an embedding $\iota\colon M\into A$ in a smooth projective ambient space $A$. Then writing $\iota_{*\!}\Ob_M$ as the quotient of a locally free sheaf $E$ over $A$, we get an embedding of $N$
$$
N=\Spec\Sym\udot\Ob_M\ \subset\ \Spec\Sym\udot E=\mathrm{Tot}\;(E^*)
$$
in a vector bundle $E^*$ over $A$.

Let $\widetilde A:=\,$Tot$\;(E^*)\to A$ denote the resulting smooth total space. Its $\C^*$-action induces an Euler vector field $v$ which restricts on $N\subset\widetilde A$ to the Euler vector field defining the cosection \eqref{cose}. Intepreting the restriction map
\beq{surj}
\Omega_{\widetilde A}\big|\_N\To\Omega_N\To0
\eeq
as a surjection from the vector bundle $\Omega_{\!\widetilde A}|\_N$ to the obstruction sheaf $\Ob_N=h^1\big((F\udot)^\vee\big)$, standard Behrend-Fantechi obstruction theory (e.g. \cite[Section 5]{BF}, \cite[Section 2]{BehDT}) gives a normal cone
$$
C_N\,\subset\ \mathrm{Tot}\;(\Omega_{\!\widetilde A})\big|\_N
$$
by pulling back the intrinsic normal cone from $h^1/h^0\big((F\udot)^\vee\big)$. It is \emph{conic Lagrangian} inside the total space of $\Omega_{\!\widetilde A}$ \cite[Theorem 4.9]{BehDT}, and its intersection with the zero section $\widetilde A$ defines $[N]^{\vir}$, but the noncompactness of $N$ makes this uninteresting.

So we perturb the 0-section of $\Omega_{\!\widetilde A}$ to make the intersection compact, and to localise it near $M$. Pick a hermitian metric $|\,\cdot\,|$ on $E^*\to A$, and consider it as a function
$$
r\colon\widetilde A\To\R, \qquad r(e)=|e|,
$$
measuring the size of $e\in E^*$ up the fibres of $E^*\to A$.

For $\epsilon>0$ let $\psi\colon[0,\infty)\to[0,\infty)$ be any smooth function satisfying
\beq{conds}
\psi(x)=\left\{\!\!\begin{array}{cc} 0 & x\le1, \\ \epsilon x & x\ge2, \end{array}\right.
\eeq
and perturb the 0-section of $\Omega_{\!\widetilde A}$ to the (non-holomorphic!) graph of $d\psi(r)$:
\beq{graph}
\Gamma_{d\psi(r)}\ \subset\ \Omega_{\!\widetilde A}.
\eeq

\begin{lem} If $M$ is compact then the intersection of $\Gamma_{\!d\psi(r)}$ with the cone $C_N$ is compact. In particular, their topological intersection is defined.
\end{lem}

\begin{proof}
Compose the cosection map \eqref{cose} with the surjection \eqref{surj},
$$
\Omega_{\widetilde A}\big|\_N\To\Omega_N\To\O_N.
$$
By construction the composition is (the restriction to $N$ of)
contraction with the Euler vector field $v$ on $\widetilde A=\,$Tot$\;(E^*)\to A$.

Restricted to the graph \eqref{graph} the composition defines a function (smooth rather than holomorphic, since the graph is not holomorphic). By calculation the function is 
$$D_v(\psi(r))=\epsilon r$$ 
outside the neighbourhood $\{r\le2\}$ of $A\subset$\,Tot$(\;E^*)$. Since this is nonzero, the graph $\Gamma_{\!d\psi(r)}$ does \emph{not} intersect the kernel $K$ of \eqref{kern}. But $C_N$ \emph{does} lie in $K$, so their intersection is empty outside the compact neighbourhood $\{r\le2\}$ of $A\subset\widetilde A$.
\end{proof}

By \cite[Corollary 4.15]{BehDT} the cone $C_N$ is the characteristic cycle $C\;\!C(\chi^N)$ of the constructible Behrend function $\chi^N\colon N\to\Z$ (extended by zero to $\widetilde A$). Its intersection with $\Gamma_{\!d\psi(r)}$ is compact, and the sets $\{e\in \widetilde A: \psi(r(e))\le t\}$ are also compact for all $t$ by the properness of $\psi(r)$. 
Therefore we can apply the Kashiwara index theorem to any constructible sheaf $\mathcal{F}$ on $N$ whose pointwise Euler characteristic is $\chi^N$. The characteristic cycle of $\mathcal F$, denoted by $\widetilde{S\;\!S\,}\!(\mathcal F)$ in \cite{Ka}, then equals $C\;\!C(\chi^N)=C_N$ so by \cite[Theorem 4.2]{Ka},\footnote{Kashiwara uses the orientation induced by the real symplectic form $\omega=d\theta=\sum_idp_idq_i$ on Tot\;$(\Omega_{\widetilde A})$, where $\theta=\sum_ip_idq_i$ is the canonical real one-form. Here $(q_i)_{i=1}^n$ are local \emph{real} coordinates on $\widetilde A$, and $n=2\dim_{\C}\widetilde A$ is the real dimension of $\widetilde A$.  Since $(-1)^{n(n+1)/2}\bigwedge_i(dp_i\wedge dq_i)=\bigwedge_i dq_i\wedge\bigwedge dp_i$ gives the standard complex orientation, our intersection $\widetilde{S\;\!S\,}\!(\mathcal F)\cdot\Gamma_{d\psi(r)}$ differs from his by the sign $(-1)^{n(n+1)/2}$.}
$$
C_N\cdot\Gamma_{d\psi(r)}\ =\ \widetilde{S\;\!S\,}\!(\mathcal F)\cdot\Gamma_{d\psi(r)}\ =\ \chi(\mathcal F)\ =\ e\big(N,\chi^N\big).
$$
And, for $\epsilon\ll1$, the graph $\Gamma_{d\psi(r)}$ is a \emph{small perturbation} of the 0-section in the sense of \cite[Appendix A]{KL}. Therefore \cite[Proposition A.1]{KL} equates $C_N\cdot\Gamma_{d\psi(r)}$ with $[N]^{\vir}_{\mathrm{loc}}$ (after pushforward to the neighbourhood $\{|e|\le2\}$).
We conclude that
\[
e_4(M)\,=\,\int_{[N]^{\vir}_{\mathrm{loc}}}1\,=\,C_N\cdot\Gamma_{d\psi(r)}\,=\,
e\big(N,\chi^N\big)\,=\,e\big(M,\chi^N|_M\big)\,=\,e_3(M).
\]

\section{Ordinary signed Euler characteristic} \label{ordinary}

The equality of $e_3(M)$ and $(-1)^{\vd\;}e(M)$ was explained to us by Davesh Maulik. We work locally in the model of Proposition \ref{locCS}.
That is, we have a smooth ambient space $A$, a bundle $E\rt\pi A$ and the section $s\in\Gamma(E)$ cutting out $M\subset A$. Then $s$ defines a function $\widetilde s$ on Tot$\;(E^*)$, linear on the fibres,
\beq{formu}
\widetilde s\;(a,e)=\langle e,s(a)\rangle,
\eeq
whose critical locus is $N=\Crit(\widetilde s\;)\subset\mathrm{Tot}\;(E^*)$.

Recall Behrend's formula \cite[Section 1.2]{BehDT}
\beq{PW}
\chi^N(p)=(-1)^{\dim\mathrm{Tot}\;(E^*)}\big(1-e(F_p)\big)
\eeq
for the Kai function of $N$. Here
$$
(-1)^{\dim\mathrm{Tot}\;(E^*)}=(-1)^{\dim A-\rk E^*}=(-1)^{\vd}
$$
and
\beq{str}
F_p\ =\ \widetilde s^{\,\;-1}(\delta)\cap B_\epsilon(p), \qquad 0<|\delta|\ll\epsilon\ll1,
\eeq
is the Milnor fibre of $\widetilde s$ at $p$. The formula \eqref{PW} extends from $p\in N$ to any $p\in\;$Tot$(E^*)$, giving 0 outside $N=\Crit(\widetilde s\;)$.

We sketch Maulik's idea of working relative to $A$. The key fact is that by \eqref{formu} the fibre of $\widetilde s^{\,\;-1}(\delta)$ over $a\in A$ is an affine space if $s(a)\ne0$ and empty if $s(a)=0$. Therefore pushing \eqref{PW} down $N\to M\subset A$ (in the sense of Euler characteristic) gives
\beq{resu}
(-1)^{\vd}\mathrm{\ \ over\ }s^{-1}(0)=M\subset A
\eeq
and $0$ elsewhere. Since $N\take M$ has a free $S^1$-action (under which $\widetilde s$ is invariant) it contributes nothing to the pushdown, so we should find that
$$
\chi^N\big|_M\ =\ (-1)^{\vd},
$$
and hence $e_3(M)=(-1)^{\vd}e(M).$

This sketch can be made to work because the $\C^*$-action on $E^*$ acts with weight 1 on $\widetilde s$, allowing one to take $0<|\delta|\ll1$ in \eqref{str} \emph{uniformly} over the noncompact fibres of $E^*\to A$. For full details we refer to \cite[Theorem A.1]{Dave}. In fact Davison proves the much more sophisticated result that the pushdown of the perverse sheaf of vanishing cycles on $N$ is a shift of the constant sheaf on $M$; taking Euler characteristics gives \eqref{resu}. For completeness we give an elementary proof from first principles at the level of Euler characteristics.

\begin{prop}
The Kai function of $N$ is the constant $(-1)^{\vd}$ on $M$, $$\chi^N\big|_M\equiv(-1)^{\vd}.$$
\end{prop}

\begin{proof}
We work at a point $p=(a_0,0)\in M\subset N\subset\mathrm{Tot}(E^*\to A)$. Shrinking $A$ in the analytic topology, we take it to be an open set of a vector space on which we fix a hermitian metric, with $a_0$ the origin. We may also assume that $E$ is trivial on this open set, and give it the trivial hermitian metric. Then
$$
F_p=\Big\{(a,e)\in\mathrm{Tot}\;(E^*)\colon|a|^2+|e|^2\le\epsilon^2,\ \langle e,s(a)\rangle=\delta\Big\}.
$$
Over a fixed $a\in A$ the fibre of $F_p$ is 
\beq{check}
\big\{e\colon\langle e,s(a)\rangle=\delta\mathrm{\ and\ }|e|^2\le\epsilon^2-|a|^2\big\}.
\eeq
The smallest $e$ solving $\langle e,s(a)\rangle=\delta$ has $|e|=|\delta|\big/|s(a)|$, so \eqref{check} is nonempty if and only if $|\delta|^2\big/|s(a)|^2\le\epsilon^2-|a|^2$, if and only if
\beq{check2}
f(a):=|s(a)|^2\big(\epsilon^2-|a|^2\big)\,\ge\ |\delta|^2.
\eeq
Since \eqref{check} is the intersection of a ball with an affine space, it is either contractible or empty.
Therefore $F_p$ is homotopy equivalent to the locus of $a\in A$ satisfying \eqref{check2}:
$$
F_p\ \simeq\,\big\{a\in A\colon f(a)\ge|\delta|^2\big\}.
$$
Shrinking $\delta$ if necessary, we may assume that $f$ has no critical values in the interval $\big(0,|\delta|^2\big]$. Thus
$$
F_p\,\simeq\,f^{-1}\big[|\delta|^2,\infty\big)\,\simeq\,f^{-1}(0,\infty)\,=\,
\stackrel{\circ}{B}\_\epsilon\!(a_0)\;\take\;s^{-1}(0).
$$
Since $s^{-1}(0)$ is a complex subvariety, the usual argument that Euler characteristic is motivic\footnote{By excision and Mayer-Vietoris we get the required formula, but with a correction from the Euler characteristic of the link of $s^{-1}(0)$. Its vanishing essentially follows from the vanishing of $e(S^{2n-1})$, where $n$ is the codimension of $s^{-1}(0)$.} now gives
$$
e(F_p)\ =\ 1-e\big(s^{-1}(0)\cap B_\epsilon(a_0)\big).
$$
But $s^{-1}(0)\cap B_\epsilon(a_0)$ is homeomorphic to the cone on its intersection with $\partial B_\epsilon(a_0)$ for $\epsilon\ll1$, so has Euler characteristic 1. Therefore $e(F_p)=0$.

Substituting into \eqref{PW} gives
\[
\chi^N(a_0,0)=(-1)^{\vd}. \qedhere
\]
\end{proof}
\medskip


\bibliographystyle{halphanum}
\bibliography{references}

\bigskip \noindent
{\tt y.jiang@ku.edu} \medskip

\noindent Department of Mathematics \\
\noindent University of Kansas \\
\noindent 405 Jayhawk Blvd \\
\noindent Lawrence, KS 66045. USA

\bigskip \noindent
{\tt richard.thomas@imperial.ac.uk} \medskip

\noindent Department of Mathematics \\
\noindent Imperial College London\\
\noindent London SW7 2AZ. UK

\end{document}